\documentclass{article}
\usepackage{graphicx} 

\usepackage{authblk}
\usepackage{tikz}
\usepackage{epsfig}
\usepackage{caption}
\usepackage{amsfonts}
\usepackage{amssymb}
\usepackage{mathrsfs}
\usepackage{amsmath}
\usepackage{amsthm}
\usepackage{enumerate}
\usepackage{graphics}
\usepackage{pict2e}
\usepackage{cite}
\usepackage{subfigure}
\usepackage[subnum]{cases}
\usepackage{multirow}
\usepackage{comment}
\usepackage{xcolor}
\usepackage{cleveref}
\usepackage{comment}

\usepackage{color}

\makeatletter

\newcommand{\KS}{\mathit{KS}} 
\newcommand{\SSG}{\mathit{SS}} 
\newcommand{\HKS}{\widehat{\KS}}
\newcommand{\HSS}{\widehat{\SSG}}

\newtheorem{theorem}{Theorem}[section]

\newtheorem{definition}[theorem]{Definition}

\newtheorem{lemma}[theorem]{Lemma}

\makeatletter
\newcommand\figcaption{\def\@captype{figure}\caption}
\newcommand\tabcaption{\def\@captype{table}\caption}
\makeatother

\newtheorem{conjecture}[theorem]{Conjecture}

\title{An Erd\H{o}s-Gallai conjecture for signed graphs}
\author{Lujia Wang \thanks{ljwang@zjnu.edu.cn}}
\date{}
\affil{Zhejiang Normal University}
    
\begin{document}

\maketitle
	
\begin{abstract}
For every natural number $p$, we show that the maximum negative girth among the class of signed graphs on $n$ vertices with balanced chromatic number at least $p$ is between $(1/e)n^{1/(p-1)}$ and $2(p-1) n^{1/(p-1)}$. This extends a conjecture of Erd\H{o}s and Gallai to signed graphs.

\end{abstract}



\section{Introduction}
Let $\lambda(n,p)$ be the minimum integer $\ell$ such that every $n$-vertex graph $G$ with chromatic number at least $p$ contains an odd cycle of length at most $\ell$. In other words, $\lambda(n,p)$ is the maximum odd girth of the class of graphs on $n$ vertices that has the chromatic number at least $p$.

Erd\H{o}s and Gallai\cite{E79} conjectured that for fixed $p$, $\lambda(n,p)$ grows asymptotically at a rate of $\Theta(n^{1/(p-2)})$. The conjecture is verified in two directions by Lov\'{a}sz\cite{L77} (independently by Schrijver\cite{S78}) and Kierstead, Szemer\'{e}di, and Trotter\cite{KST84}, who showed a lower bound and an upper bound, respectively.

\begin{theorem}\label{thm:EGconjecture}
    For every fixed $p$, 
    $$\lambda(n,p)=\Theta(n^{1/(p-2)}).$$
\end{theorem}

A \emph{signed graph} $(G,\sigma)$ is a graph $G$ with a \emph{signature function} $\sigma:E(G)\to\{-,+\}$. When the signature function is not specified, a signed graph is often denoted by $\widehat{G}$. \emph{Switching} a vertex $v\in V(G)$ is the operation to flip all signs on the edges of $\widehat{G}$ incident to $v$. A subset $B\subseteq V(G)$ of vertices is \emph{balanced} if the signature on $\widehat{G}[B]$ can be made all positive by a series of switchings; otherwise it is \emph{unbalanced}.

A \emph{balanced $p$-coloring} of $\widehat{G}$ is a function $f:V(G)\to[p]$ such that $f^{-1}(i)$ is a balanced set in $\widehat{G}$ for all $i\in [p]$. The \emph{balanced chromatic number} $\chi_b(\widehat{G})$ is the minimum integer $p$ such that $\widehat{G}$ admits a balanced $p$-coloring. The balanced chromatic number of a signed graph is invariant under switchings.

A closed walk in $\widehat{G}$ is positive (negative, respectively) if the product of its edges is positive (negative, respectively). Harary\cite{H53} showed that a set of vertices is balanced in a signed graph if and only if it induces no negative cycles. Hence, for signed graphs, the unbalanced (equivalently, negative) cycles play the role of coloring obstacles for balanced coloring, resembling the case of odd cycles for ordinary coloring. The \emph{negative girth}, $g_-(\widehat{G})$, of a signed graph $\widehat{G}$ is the length of a shortest negative cycle in $\widehat{G}$.

Let $\lambda_s(n,p)$ be the maximum negative girth among the class of signed graphs on $n$ vertices with the balanced chromatic number at least $p$. We show a result that is parallel to the Erd\H{o}s-Gallai conjecture for signed graphs.

\begin{theorem}\label{thm:signedErdosGallai}
    For every fixed $p$, $\lambda_s(n,p)=\Theta(n^{1/(p-1)}).$ In particular,
    $$(1/e)n^{1/(p-1)} \leq \lambda_s(n,p)\leq 2(p-1) n^{1/(p-1)}.$$
\end{theorem}

\section{The lower bound and Kneser signed graphs}
Two vertices $x,y$ in a signed graph $(G,\sigma)$ are \emph{antitwins} if $N(x)=N(y):=Z$ and $\sigma(xv)=-\sigma(yv)$ for all $v\in Z$. A \emph{double switching graph} is a signed graph where each vertex has exactly one antitwin. A \emph{reduction} of a double switching graph is the process of deleting exactly one vertex from each pair of antitwins. Switching a vertex of the reduced graph is the same as replacing it with its antitwin (hence, it is another reduction of the same signed graph).

Let positive integers $n$ and $k$ satisfy $n\geq k$. $[\pm n]:=\{\pm 1,\pm 2,\ldots,\pm n\}$. A set $S$ is a \emph{signed $k$-subset of $[n]$} if $S\subseteq[\pm n]$, $|S|={k}$, and $S\cap (-S)=\emptyset$. The family of all $k$-subsets of $[n]$ is denoted by $\binom{[n]}{k_{\pm}}$.

The \emph{Kneser signed graph} $\KS(n,k)$ is the signed graph defined on the vertex set $\binom{[n]}{k_{\pm}}$ where $AB$ is a positive (respectively, negative) edge if $A\cap (-B)=\emptyset$ (respectively, $A\cap B=\emptyset$). The subgraph of $\KS(n,k)$ induced by alternating $k$-subsets of $[n]$ is called the \emph{Schrijver signed graphs}, denoted by $\SSG(n,k)$. Both graphs are double switching graphs, hence have a reduction by deleting one of the vertices from each pair of antitwins. The corresponding reductions are denoted by $\HKS(n,k)$ and $\HSS(n,k)$, respectively.

Kuffner et al.\cite{KNWYZZ25+} proved that the balanced chromatic number $\chi_b(\KS(n,k))=\chi_b(\SSG(n,k))=n-k+1$, and that $\HSS(n,k)$ is a vertex-critical subgraph of the $\HKS(n,k)$.

We first determine the negative girth of the Kneser signed graphs. 

\begin{lemma}\label{lem:negGirth}
    For all positive integers $n$ and $k$ such that $n\geq k$,
    $$g_-(\SSG(n,k))\geq g_-(\KS(n,k))=\begin{cases}
    1+\left\lceil\frac{k}{n-k}\right\rceil &\text{if} \quad n> k,\\
    \infty &\text{if} \quad n=k.
    \end{cases}
    $$
\end{lemma}
\begin{proof}
    The first inequality is clear, as Schrijver signed graphs are induced subgraphs of Kneser signed graphs. So, it is enough to determine the negative girth of Kneser signed graphs. 
    
    The Kneser signed graph $\KS(k,k)$ is balanced, so the negative girth is $\infty$. If $n\geq 2k$, then $g_-(\KS(n,k))=2$, since there exist unbalanced digons. So, it suffices to deal with the case where $k< n< 2k$.
    
    Since $\KS(n,k)$ is a double switching graph, every negative cycle can be switched into an \emph{inconsistent cycle}, i.e., a negative cycle with exactly one negative edge. So we think of the length of a shortest inconsistent cycle.

    Let $A_1A_2\ldots A_\ell A_1$ be a shortest inconsistent cycle in $\KS(n,k)$, where $A_iA_{i+1}$ is positive for $i\in [\ell-1]$, and $A_1A_\ell$ is negative.

    For each $i\in[\ell-1]$, since $A_iA_{i+1}$ is positive, $A_{i}$ is disjoint from $-A_{i+1}$, which implies that $|A_{i+1}\setminus A_i|\leq n-k.$ Hence, we have
    $$|A_\ell\setminus A_1|\leq |A_2\setminus A_1|+|A_3\setminus A_2|+\cdots+|A_\ell\setminus A_{\ell-1}|\leq (\ell-1)(n-k).$$

    On the other hand, $A_\ell$ and $A_1$ are disjoint, since $A_1A_\ell$ is a negative edge. So, $$k=|A_\ell|=|A_\ell\setminus A_1|\leq(\ell-1)(n-k),$$
    which gives 
    $$\ell\geq 1+ \left\lceil\frac{k}{n-k}\right\rceil.$$

    Finally, we construct a negative cycle of length $1+\lceil\frac{k}{n-k}\rceil$ in $\KS(n,k)$ as follows. First, the ground set $[\pm n]$ as ${1,2,\ldots,n, -1,-2,\ldots,-n}$. Start with $A_1=[k]$. When $A_i$ is defined, let $A_{i+1}$ be obtained by shifting all elements of $A_i$ to the right by $n-k$, until we get $A_\ell$, where $\ell=1+\lceil\frac{k}{n-k}\rceil$. It is easy to check that this results in an inconsistent cycle of the desired length, and the proof is complete.
\end{proof}
We remark that the inequality in \Cref{lem:negGirth} can be strict. For example, $g_-(\SSG(6,4))=4$ but $g_-(\KS(6,4))=3$.

\paragraph{Proof of a lower bound of \Cref{thm:signedErdosGallai}.}Now we construct an $n$-vertex balanced-$p$-chromatic signed graph with negative girth at least $(1/e)n^{1/(p-1)}$.

Choose the integer $k$ to be the largest such that $\binom{p+k-1}{k}\leq n$. The signed graph $\widehat{G}$ is defined as $\HSS(p+k-1,k)$ together with $n-\binom{p+k-1}{k}$ isolated vertices. We have $\chi_b(\widehat{G})=\chi_b(\HSS(p+k-1,k))=p$, and the negative girth of $\widehat{G}$ is at least $1+\lceil\frac{k}{p-1}\rceil$. We are left to show that $1+\lceil\frac{k}{p-1}\rceil\geq (1/e)n^{1/(p-1)}$.

By maximality of $k$,
$$n<\binom{p+k}{k+1}=\binom{p+k}{p-1}<\left(\frac{e(p+k)}{p-1}\right)^{p-1}\leq \left(e\left(1+\left\lceil\frac{k}{p-1}\right\rceil\right)\right)^{p-1}.$$
This implies that
$$\frac{n^{\frac{1}{p-1}}}{e}\leq 1+\left\lceil\frac{k}{p-1}\right\rceil,$$
and the proof is complete.

\section{The upper bound}
For the upper bound, we follow the approach of Kierstead et al.\cite{KST84}, and prove the signed version of their main theorem. Note that the upper bound of \Cref{thm:signedErdosGallai} follows from \Cref{thm:signedKST} as a special case with $p=1$.

\begin{theorem}\label{thm:signedKST}
    For each pair of positive integers $p$ and $q$, if $\widehat{G}$ is a signed graph on $n$ vertices that contains no subgraph $\widehat{H}$ whose balanced chromatic number is greater than $p$ and whose radius in $\widehat{G}$ is at most $qn^{1/q}$, then the balanced chromatic number of $\widehat{G}$ is at most $pq$.
\end{theorem}
Another way to formulate this statement is that if the $qn^{1/q}$-neighborhood of every vertex of $\widehat{G}$ is balanced $p$-colorable, then $\widehat{G}$ is balanced $pq$-colorable.

In the following discussion, we express the integer of the form $\lceil n^{1/k}\rceil^m$ by $n^{m/k}$ for simplicity. For every pair of vertices $x,y$ in a signed graph $\widehat{G}$, the \emph{distance} between $x$ and $y$ in $\widehat{G}$, denoted by $d_{\widehat{G}}(x,y)$ is the length of the shortest $xy$-path in the underlying graph $G$.

If $S\subseteq V(G)$,
let $d_{\widehat{G}}(x,S):=\min_{y\in S}d_{\widehat{G}}(x,y)$, and for every subgraph $\widehat{H}$ of $\widehat{G}$, the distance $d_{\widehat{G}}(x,\widehat{H}):=d_{\widehat{G}}(x,V(H))$.
The \emph{radius} of $\widehat{H}$ in $\widehat{G}$ is defined as 
$$R_{\widehat{G}}(\widehat{H}):=\min_{x\in V(H)}\max_{y\in V(H)}d_{\widehat{G}}(x,y).$$

When the ambient graph $\widehat{G}$ is clear from the context, it is omitted from the subscript. We will often abuse the notation of a subset of vertices of a given signed graph and the subgraph it induces. For example, for $Q\subseteq V(G)$, the induced subgraph $\widehat{G}[Q]$ is referred to as the subgraph $Q$ of $\widehat{G}$.

\begin{definition}
    An $(\alpha,\beta)$-obstruction in a signed graph $\widehat{G}$ is a subgraph $Q$ with $|Q|\geq \alpha$ and $R(Q)\leq \beta$.
\end{definition}

To prove \Cref{thm:signedKST}, it suffices to prove the following stronger statement.

\begin{lemma}\label{lem:signedKST}
    Let $\widehat{G}$ be a signed graph on $n$ vertices such that for any subgraph $\widehat{H}$ of $\widehat{G}$, if $R_{\widehat{G}}(\widehat{H})\leq qn^{1/q}$, then $\chi_b(\widehat{H})\leq p$. If $0\leq \ell\leq q$, $W\subseteq \widehat{G}$, and $\chi_b(W)>\ell p$, then $W$ contains an $(n^{\ell/q},\ell n^{1/q})$-obstruction.
\end{lemma}

We first derive \Cref{thm:signedKST} from this lemma.
\begin{proof}[Proof of \Cref{thm:signedKST}]
    Suppose $\chi_b(\widehat{G})>pq$. Take $\ell=q$ and $W=\widehat{G}$ in \Cref{lem:signedKST}. Then, there is an $(n,qn^{1/q})$-obstruction in $\widehat{G}$, which must be $\widehat{G}$ itself. But this means $R(\widehat{G})\leq qn^{1/q}$, which, by the assumption, implies that $\chi_b(\widehat{G})\leq p$. This is a contradiction.
\end{proof}
\begin{proof}[Proof of \Cref{lem:signedKST}]
    We apply the induction on $\ell$. For the case $\ell=0$, since $W\neq \emptyset$, we can find an arbitrary vertex $v\in W$ which is a $(1,0)$-obstruction. (One can also treat the case $\ell =1$ as the base case.)

    Now, assume the claim is true for $\ell=m$, and consider the case $\ell=m+1\leq k$. Let $H\subseteq W$ be a balanced $p$-chromatic subgraph with the maximum number of vertices. 

    Consider the subgraph $W'=W\setminus H$. Clearly, $\chi_b(W')\geq (m+1)p-p=mp$. So by the induction hypothesis, $W'\setminus H$ contains an $(n^{m/q},mn^{1/q})$-obstruction $P_0$. Note that $H\cap P_0=\emptyset$. For each $j>0$, let $P_j=\{x\in H: d(x,P_0)=j\}$. Let $Q=\bigcup_{j\leq n^{1/q}}P_j$. Clearly, $R(Q)\leq n^{1/q}+R(P_0)\leq 2\ell n^{1/q}.$

    \textbf{Claim.} For any $j< n^{1/q}$, $|P_j|\geq n^{m/n}$. 

    Suppose, on the contrary, $|P_j|< n^{m/n}$. Then, $H'=H\setminus P_j\cup P_0$ has a cardinality greater than $H$. It is hence enough to obtain a contradiction by showing that $\chi_b(H')\leq p$. Partition $H'$ into $(H_0,H_1)$ where $H_0=\bigcup_{i\leq j-1}P_i$ and $H_1=\bigcup_{i\geq j+1}P_i$. Since $R(H_0)\leq qn^{1/q}$, $\chi_b(H_0)\leq p$. Since $H_1\subseteq H$, $\chi_b(H_1)\leq p$. Because there are no edges joining $H_0$ and $H_1$, we have $\chi_b(H')\leq p$.

    By the claim, it is clear that $|Q|\geq n^{1/q}\cdot n^{m/q}=n^{\ell/q}$, and that $Q$ is an $(n^{\ell/q}, \ell n^{1/q})$-obstruction. This completes the proof.
\end{proof}

\section{Balanced 3-chromatic signed graphs}
\Cref{thm:signedKST} with $p=1$ and $q=2$ has the consequence that every $n$-vertex balanced $3$-chromatic signed graph contains a negative cycle of length at most $4\sqrt{n}$. We now improve this by finding a negative cycle of length shorter than $2\sqrt{n-1}+1$. The approach is similar to that of \Cref{thm:signedKST}, and the parallel graph case (for $4$-chromatic graphs) can be found in Nilli \cite{N99} and Jiang \cite{J01}. Recently, Jiang's result received a minor improvement by Naserasr et al. \cite{NWZpc}.

\begin{theorem}\label{thm:balanced3XGraphs}
    Every balanced $3$-chromatic graph on $n$ vertices contains a negative cycle of length less than $2\sqrt{n-1}+1$.
\end{theorem}

\begin{proof}
    Let $\widehat{G}$ be a balanced $3$-chromatic graph on $n$ vertices. Suppose $g_-(\widehat{G})=\ell$.

    Let $B$ be a maximum balanced set in $\widehat{G}$. Since $\chi_b(\widehat{G)}=3$, $\widehat{G}-B$ contains a negative cycle. Let $C$ be a shortest negative cycle in $\widehat{G}-B$ (clearly, it is an induced cycle whose length is at least $\ell$). Let $v$ be a vertex in $C$. For each integer $i\geq 0$, define $V_i=\{u\in V(G): d(u,v)=i\}$.

    \textbf{Claim.} For $0\leq i\leq \lfloor(\ell-2)/2\rfloor$, $V_i$ is a balanced set.

    \textbf{Proof of Claim.} To see this, we switch to make all the edges between $V_{j-1}$ and $V_j$ positive for $j=1,2,\ldots, i$. The resulting $V_i$ induces no negative edges; otherwise, the negative girth of $\widehat{G}$ is less than $\ell$. Hence, $V_i$ is balanced.

    Now, for each $i$, $1\leq i\leq \lfloor(\ell-2)/2\rfloor$, define $H_i$ to be the subgraph of $\widehat{G}$ induced by 
    $$U_i:=\left(\bigcup_{0\leq j<i}V_j\right)\cup \left(\bigcup_{j>i}(V_j\cap B)\right).$$

    Clearly, $U_i$ is a balanced set. By the choice of $B$, we have $|B|\geq |U_i|$. In particular, $|B-U_i|\geq |U_i-B|$.

    Note that $U_i$ contains $2i-1$ vertices of $C$ that are within distance $i-1$ from $v$. These vertices do not belong to $B$. So we have $|U_i-B|\geq 2i-1$. On the other hand, $B-U_i=V_i\cap B$. It follows that $|V_i\cap B|=|B\cap U_i|\geq |U_i-B|\geq 2i-1$. Hence, we have

    $$|B|\geq \sum_{i=1}^{\lfloor(\ell-2)/2\rfloor} |V_i\cap B|\geq \sum_{i=1}^{\lfloor(\ell-2)/2\rfloor} 2i-1= \left(1+2\left\lfloor\frac{\ell-2}{2}\right\rfloor-1\right)\left\lfloor\frac{\ell-2}{2}\right\rfloor/2=\left\lfloor\frac{\ell-2}{2}\right\rfloor^2.$$

    Due to this and since $C$ is disjoint from $B$, we have

    $$n\geq \ell+|B|=\ell+\left\lfloor\frac{\ell-2}{2}\right\rfloor^2.$$

    Hence, $\ell<2\sqrt{n-1}+1$.
\end{proof}

\section{Concluding remarks}
\paragraph{5.1.} The bound in \Cref{thm:signedKST} is asymptotically tight for $p=1$, which gives the main result of this paper (\Cref{thm:signedErdosGallai}). It remains open whether it is best possible for larger $p$.

\paragraph{5.2.} Equivalent to the problem of determining $\lambda_s(n,p)$, one can ask about $n_s(\lambda,p)$, the minimum number of vertices of a signed graph with balanced chromatic number at least $p$ and negative girth at least $\lambda$. It is easily observed that $n_s(3,p)=2p-1$, and the only example is $(K_{2p-1},-)$ (the notation $(G,-)$ refers to the signed graph on the underlying graph $G$ whose edges are all negative).

The next open case is $n_s(4,3)$, and from the proof of \Cref{thm:balanced3XGraphs}, $n_s(4,3)\geq 5$. Note that for graphs, the minimum number of vertices of a $4$-chromatic graph with odd girth $4$ is $11$, and the Gr\"{o}tzsch graph is the only example that achieves this bound. We make the following conjecture for signed graphs.

\begin{conjecture}
    $n_s(4,3)=13$ and the only signed graph with $13$ vertices, balanced chromatic number $3$ and negative girth $4$ is the signed graph depicted in \Cref{fig:13}.
\end{conjecture}

	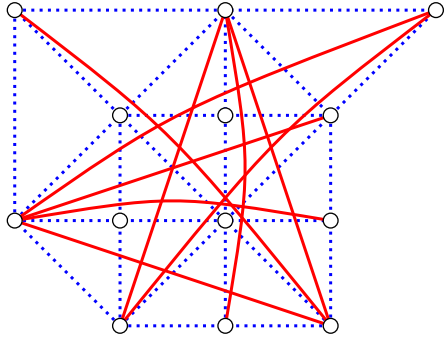
\begin{figure}[!htbp]
		\centering
		
		\begin{tikzpicture}[scale=.7]		
			\draw [line width=0.4mm, dotted, blue] (-4,0) to (2,0);
            \draw [line width=0.4mm, dotted, blue] (0,-2) to (0,4);
            \draw [line width=0.4mm, dotted, blue] (-4,0) to (-4,4);
            \draw [line width=0.4mm, dotted, blue] (2,-2) to (2,2);
            \draw [line width=0.4mm, dotted, blue] (-2,-2) to (-2,2);
            \draw [line width=0.4mm, dotted, blue] (-2,2) to (2,2);
            \draw [line width=0.4mm, dotted, blue] (-2,-2) to (2,-2);
            \draw [line width=0.4mm, dotted, blue] (2,-2) to (-4,4);
            \draw [line width=0.4mm, dotted, blue] (-2,-2) to (4,4);
            \draw [line width=0.4mm, dotted, blue] (-4,4) to (4,4);

            \draw [line width=0.4mm, dotted, blue] (-2,-2) to (-4,0) to (-2,2) to (0,4) to (2,2);

			\draw [line width=0.4mm, red] (2,2) to (-4,0) to (2,-2);
			\draw [line width=0.4mm, red] (-2,-2) to (0,4) to (2,-2);
			\draw [line width=0.4mm, red] (-4,0) .. controls (-1,2.1) .. (4,4);
            \draw [line width=0.4mm, red] (0,-2) .. controls (0.5,1) .. (0,4);
            \draw [line width=0.4mm, red] (2,-2) .. controls (-1,1.7) .. (-4,4);
            \draw [line width=0.4mm, red] (-2,-2) .. controls (1,1.7) .. (4,4);
            \draw [line width=0.4mm, red] (-4,0) .. controls (-1,0.5) .. (2,0);

			\draw [fill=white,line width=0.5pt] (0,-2) node[above] {} circle (4pt); 
			\draw [fill=white,line width=0.5pt] (0,0) node[above] {} circle (4pt); 
			\draw [fill=white,line width=0.5pt] (0,2) node[above] {} circle (4pt); 
			\draw [fill=white,line width=0.5pt] (2,0) node[above] {} circle (4pt); 
            \draw [fill=white,line width=0.5pt] (-2,0) node[above] {} circle (4pt); 
			\draw [fill=white,line width=0.5pt] (2,2) node[above] {} circle (4pt); 
			\draw [fill=white,line width=0.5pt] (2,-2) node[above] {} circle (4pt); 
			\draw [fill=white,line width=0.5pt] (-2,2) node[above] {} circle (4pt); 
            \draw [fill=white,line width=0.5pt] (-2,-2) node[above] {} circle (4pt);

			\draw [fill=white,line width=0.5pt] (-4,0) node[above] {} circle (4pt); 
			\draw [fill=white,line width=0.5pt] (-4,4) node[above] {} circle (4pt); 
			\draw [fill=white,line width=0.5pt] (0,4) node[above] {} circle (4pt); 
			\draw [fill=white,line width=0.5pt] (4,4) node[above] {} circle (4pt); 
			
		\end{tikzpicture}
		\caption{A balanced $3$-chromatic signed graph with negative girth $4$ on $13$ vertices.}
		\label{fig:13}	

    \end{figure}
 The construction in \Cref{fig:13} is a ``generalized Mycielskian of a negative $4$-cycle". For more discussion on this construction and the connection with topological methods, the reader is referred to \cite{NWZ25}.

\paragraph{5.3.} The upper bound of \Cref{thm:signedErdosGallai} does not imply \Cref{thm:EGconjecture} but only proves an upper bound of $\lambda(n,p)=O(n^{2/(p-2)})$. To see this, let $G$ be a graph on $n$ vertices with the chromatic number at least $p$. Then, the basic property of balanced coloring gives $\chi_b((G,-))\geq \lceil p/2\rceil$. By \Cref{thm:signedErdosGallai}, $(G,-)$ contains a negative cycle of length at most $O(n^{2/(p-2)})$ which corresponds to an odd cycle in $G$.

On the other hand, given a signed graph $\widehat{G}$ with $\chi_b(\widehat{G})\geq p$, the subgraph of $\widehat{G}$ induced by the negative edges (considered as a graph) has the chromatic number at least $p$. Hence, by \Cref{thm:EGconjecture}, we find an odd cycle of length at most $O(n^{1/{(p-2)}})$ in this subgraph, which is a negative cycle in $\widehat{G}$.

It is of interest to figure out whether \Cref{thm:EGconjecture} and \Cref{thm:signedErdosGallai} have direct implications.

\paragraph{Acknowledgment.} The author thanks a referee for valuable comments and for pointing out the question discussed in 5.2. 

Lujia Wang's research is supported by the Natural Science Foundation of China grant No. 12371359, and Zhejiang Normal University Postdoc Startup Funding No. ZC304022966.

\bibliographystyle{plain}

\end{document}